\documentclass[a4paper]{amsart}

\usepackage[utf8]{inputenc}
\usepackage{mathtools}
\usepackage{amssymb,amsfonts,amsmath}
\usepackage{enumitem}
\usepackage{color}
\usepackage{mathrsfs}
\usepackage{graphicx}
\usepackage{verbatim}
\usepackage{tikz-cd}
\usepackage{tikz}
\usepackage{url} 
\usepackage{mathabx}

\numberwithin{equation}{section}
\newtheorem{theorem}{Theorem}[section]
\newtheorem{corollary}[theorem]{Corollary}
\newtheorem{lemma}[theorem]{Lemma}

\newtheorem{proposition}[theorem]{Proposition}
\newtheorem{question}[theorem]{Question}

\theoremstyle{definition}
\newtheorem{definition}[theorem]{Definition}
\newtheorem{notation}[theorem]{Notation}

\theoremstyle{remark}
\newtheorem{remark}[theorem]{Remark}

\newcommand{\N}{\mathbb{N}}

\newcommand{\Z}{\mathbb{Z}}

\DeclareMathOperator{\ab}{ab}

\DeclareMathOperator{\id}{id}

\DeclareMathOperator{\Homeo}{Homeo}

\DeclareMathOperator{\supp}{supp}
\DeclareMathOperator{\Sym}{Sym}

\newcommand{\abs}[1]{\vert #1 \vert}
\newcommand\Set[2]{\{\,#1\mid#2\,\}}

\newcommand{\defeq}{\mathrel{\mathop{:}}=}

\renewcommand{\epsilon}{\varepsilon}

\providecommand{\comeddy}[1]{\textsuperscript{\textcolor{blue}{?}}{\marginpar{\begin{flushleft} \tiny \textbf{Ed.:} \textcolor{blue}{#1} \end{flushleft}}}}

\title[The Higman--Thompson groups $V_n$ are $(2,2,2)$-generated]{The Higman--Thompson groups $V_n$ are $(2,2,2)$-generated}

\author[E. Schesler]{Eduard Schesler}
\address{Karlsruhe Institute of Technology, Englerstr.\ 2, 76131 Karlsruhe, Germany}
\email{eduardschesler@googlemail.com}

\author[R. Skipper]{Rachel Skipper}
\address{University of Utah, 155 S 1400 E, Salt Lake City, UT 84112 USA}
\email{rachel.skipper@utah.edu}

\author[X. Wu]{Xiaolei Wu}
\address{Shanghai Center for Mathematical Sciences, Jiangwan Campus, Fudan University, No.2005 Songhu Road, Shanghai, 200438, P.R. China}
\email{xiaoleiwu@fudan.edu.cn}

\keywords{The Higman--Thompson groups, $(2,2,2)$-generated.}

\begin{document}
\begin{abstract}
We provide a family of generating sets $S_{\alpha}$ of the Higman--Thompson groups $V_n$ that are parametrized by certain sequences $\alpha$ of elements in $V_n$.
These generating sets consist of $3$ involutions $\sigma$, $\tau$, and $s_{\alpha}$, where the latter involution is inspired by the class of spinal elements in the theory of branch groups.
In particular this shows the existence of generating sets of $V_n$ that consist of $3$ involutions.
\end{abstract}
\maketitle

\section{Introduction}
The Thompson group $V$ first appeared in some handwritten never published notes by Richard Thompson \cite{Tho65}. Together with its subgroup $T$, they are the first known finitely presented infinite simple groups. These were generalized by Higman to the Higman--Thompson groups denoted $V_n$ \cite{Hig74} and have been extensively studied in the intervening decades. In fact, with few exceptions \cite{BM97, CapraceRemy06}, essentially all known examples of infinite, finitely presented simple groups have been modeled on, or built by extending, the constructions of $T$ or $V$. These groups naturally occur in many different contexts, such as geometric group theory \cite{BrownGeoghegan84, Brown87b}, $C^\ast$-algebras \cite{Nekrashevych04}, and others. We direct the reader to \cite{CannonFloydParry96, Hig74} for an introduction on these groups.

The group $V_n$ has a rich subgroup structure. For instance it contains every finite group as a subgroup. The Higman--Thompson groups are often thought of as an infinite analog of a finite symmetric group and thus it is an interesting question to find analogous small generating sets or presentations for $V_n$. Donoven and Harper proved that $V_n$ is $\frac{3}{2}$-generated for any $n\geq2$ \cite[Theorem 1]{DonovenHarper20}, i.e.~$V_n$ 
 is $2$-generated and every nontrivial element of $V_n$ is contained in a generating pair, a property satisfied by every finite simple group \cite{GK00}. Similarly, Bleak and Quick found a presentation for $V$ with only $2$ generators and $7$ relations \cite[Theorem 1.3]{BleakQuick2017} by exploiting a permutational description of elements of $V$. In fact, the two generators are of order $3$ and $6$. A particularly nice case is when a generating set consists of a small number of involutions, i.e.\ order $2$ elements.
 In this direction, Higman showed in \cite{Hig74} that $V$ can be generated by $4$ involutions. This was generalized to all $V_n$ in \cite[Theorem 4.1]{CorsonHughesMollerVarghese23} by Corson, Hughes, M\"uller and Varghese.
 It was therefore natural for them to ask the following question, which was brought to us by Olga Varghese, see~\cite[Question 1.3]{CorsonHughesMollerVarghese23} in the more recent version.

\begin{question}
Is $V_n$, with $n\geq 2$, $(2,2,2)$-generated? Is it
$(2,3)$-generated?
\end{question}

Recall that a group is called $(2,2,2)$-generated if it can be generated by three involutions; similarly, it is $(2,3)$-generated if it can be generated by an involution and an order $3$ element.  We answer their first question in the positive.

\begin{theorem}
    The Higman--Thompson group $V_n$ is $(2,2,2)$-generated for every $n\geq 2$.
\end{theorem}
\begin{remark}
    $V_n$ is not $(2,2)$-generated, as if so, $V_n$ would be a quotient of $\mathbb{Z}/2\bigast \mathbb{Z}/2$. But $\mathbb{Z}/2\bigast \mathbb{Z}/2$ is isomorphic to the infinite dihedral group which is solvable, while $V_n$ contains  free subgroups of rank $2$.
\end{remark}

In this context, we are also aware of related work in progress by Bleak, Donoven, Harper, and Hyde showing finitely generated simple vigorous groups are quotients of $C_2\ast C_3$ and working out explicit generators in the case of $V=V_2$. 

\subsection*{Acknowledgements.}
RS was partially supported by NSF DMS-2005297. XW was partially supported by NSFC 12326601. We thank the Shanghai Center for Mathematical Sciences, where this project started, for the hospitality.

\section{Background on $V_n$}

\noindent For the rest of this section we fix a finite set $X_n \defeq \{x_1,\ldots,x_n\}$ of cardinality $n \geq 2$, which we will think of as an alphabet. 
The Cantor set $X_n^{\N}$ will be denoted by $\mathfrak{C}_n$.
We write $X_n^{\ast} = \coprod \limits_{\ell \in \N_0} X_n^{\ell}$ to denote the set of finite words over $X_n$.
For every word $w \in X_n^{\ast}$ we refer to $w\mathfrak{C}_n$ as the set of elements in $\mathfrak{C}_n$ that have $w$ as an initial subword.
More generally, we write $A\mathfrak{C}_n \defeq \bigcup_{w \in A} w \mathfrak{C}_n$ for every subset $A \subseteq X_n^{\ast}$. We shall call a finite subset $A$ of $X_n^{\ast}$ \emph{a 
$n$-adic partition set of $\mathfrak{C}_n$} if $\mathfrak{C}_n = \bigcup_{w \in A} w \mathfrak{C}_n$ and $w\mathfrak{C}_n \bigcap w'\mathfrak{C}_n = \emptyset$ for any $w\neq w'$ in $A$.

There are a number of equivalent definitions of the group $V_n$, see e.g.~\cite{Hig74}.
The following is the one most suited for our purposes.

\begin{definition}\label{def:V_n}
The group $V_n$ is defined as the subgroup of $\Homeo(\mathfrak{C}_n)$ that consists of those homeomorphisms $\alpha$ for which there exists a $n$-adic partition set $A$ of $\mathfrak{C}_n$ and a map $f \colon A \rightarrow X_n^{\ast}$ such that $\alpha$ is given by $\alpha(w \xi) = f(w) \xi$ for every $w \in A$.
\end{definition}

\begin{notation}\label{not:support}
For every $\gamma \in V_n$, we write
\[
\supp(\gamma) = \Set{\xi \in \mathfrak{C}_n}{\gamma(\xi) \neq \xi}
\]
to denote the \emph{support of $\gamma$ in $\mathfrak{C}_n$}.
\end{notation}

\begin{notation}\label{not:gamma-vert-v}
Given an element $\gamma \in V_n$ and a word $w \in \mathfrak{C}_n$, we write $w\gamma \in V_n$ to denote the unique element with $\supp(w\gamma) \subseteq w\mathfrak{C}_n$ that is given by $w\gamma(w\xi) = w\gamma(\xi)$ for every $\xi \in \mathfrak{C}_n$.
More generally, we write $wH \defeq \Set{wh}{h \in H} \subseteq V_n$ for every subset $H \subseteq V_n$.
\end{notation}

For every permutation $\alpha \in \Sym(X_n)$, let $\dot{\alpha} \in V_n$ denote the element that is given by $\dot{\alpha}(x\xi) = \alpha(x)\xi$ for all $x \in X_n$ and $\xi \in \mathfrak{C}_n$. Similarly, we will write $\dot{\Sym}(X_n)$ for the group of all such elements of $V_n$ permuting the first level of $\mathfrak{C_n}$.
The following is a generalization of  \cite[Proposition 6.3]{BBQS22}. The same proof given there, replacing $V$ with $V_n$, works here as well.

\begin{proposition}\label{prop:maximal}
The group $\oplus_{i=1}^n x_i V_n \rtimes \dot{\Sym}(X_n)$ is maximal in $V_n$.
\end{proposition}

Higman proved in \cite[Section 5]{Hig74} that  $V_n$ is simple when $n$ is even and if $n$ is odd, it has an index $2$ simple subgroup. When $n$ is odd, the index $2$ subgroup can be described by the kernel of the sign map: given any element $g$ of $V_n$, we can represent it by a bijection between two $n$-adic  partition sets $A_1$ and $A_2$ of $\mathfrak{C}_n$. We list the elements of $A_1,A_2$ using the lexicographic order (here we fix an order on $X_n$ by setting $x_1 <x_2\cdots <x_n$),  and identify them with $1,2,\cdots, |A_1|$ using the order. This way we can associate the bijection between $A_1$ and $A_2$ with an element in $\Sym(\{1,2,\cdots |A_1|\})$. The sign of $g$ is defined to be the sign of this permutation. Note that this map is only well-defined when $n$ is odd.
Regarding this, the abelianization of $V_n$ can be described as follows.

\begin{theorem}\label{thm:simple-commutator}
 The abelianization $V_n^{\ab}$ of $V_n$ is generated by the image of $\dot{(x_1,x_2)}$.
\end{theorem}

\noindent We will also make use of the subgroup $E_n \leq V_n$ of volume-preserving homeomorphisms of $\mathfrak{C}_n$.

\begin{lemma}\label{lem:level-pres}
The group $E_n$ is locally finite.
\end{lemma}
\begin{proof}
Let $\Delta=\{\alpha_1, \dots, \alpha_d\}$ be a finite collection of volume preserving homeomorphisms.  Let $A_i \subseteq \mathfrak{C}_n$ and $f_i \colon A_i \rightarrow X_n^{\ast}$ be such that $\alpha_i$ is given by $\alpha_i(w \xi) = f_i(w) \xi$ for every $w \in A_i$. We can take a suitable refinement of $A_1, \dots, A_d$ so that, without loss of generality, we may assume that $A = A_1=\cdots=A_d=
X_n^{\ell}$ for a sufficiently large number $\ell \in \N$. In view of this, we see that the order of $\langle\alpha_i \mid i=1, \dots d \rangle$ is bounded by the order of $\Sym \Set{w\mathfrak{C}_n}{w \in X_n^{\ell}}$ which is finite.
\end{proof}

\section{Generating $V_n$ with involutions}

\noindent Let us again fix a natural number $n \geq 2$.
Recall that for every permutation $\alpha \in \Sym(X_n)$, we denote by $\dot{\alpha}$ the element in $V_n$ given by $\dot{\alpha}(x\xi) = \alpha(x)\xi$ for all $x \in X_n$ and $\xi \in \mathfrak{C}_n$.
Moreover, we fix $\sigma=(x_1,x_2) \in \Sym(X_n)$ so that $\dot{\sigma}$ is the element as shown in Figure~\ref{fig:sigma}. 

\begin{figure}[h]
\centering
\includegraphics[scale=0.4]{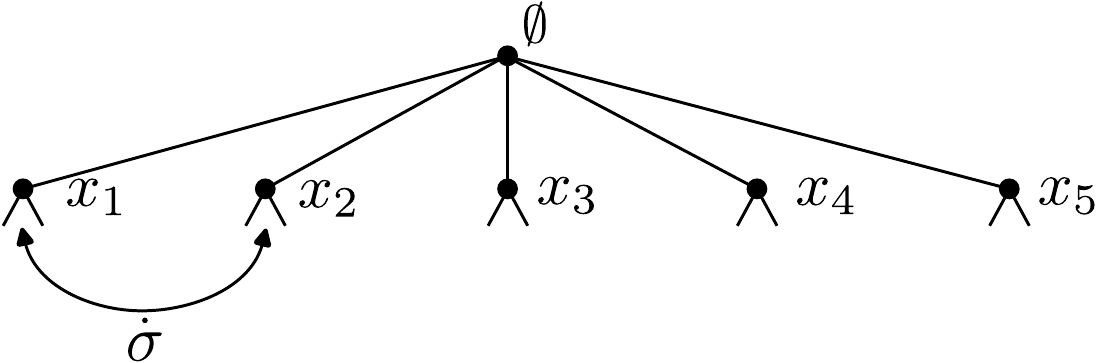}
\caption{The action of $\dot{\sigma} \in V_5$ on $\mathfrak{C}_5$.}
\label{fig:sigma}
\end{figure}

\noindent Let $\tau \in V_n$ denote the unique element with
\[
\supp(\tau)
\subseteq (\{x_1\} \times (X_n \setminus \{x_n\})) \mathfrak{C}_n
\cup (X_n \setminus \{x_1\})\mathfrak{C}_n
\]
that satisfies
\[
\tau(x_{1}x_{i}\xi) = x_{i+1}\xi
\text{ and }
\tau(x_{i+1}\xi) = x_{1}x_{i}\xi
\]
for all $1 \leq i < n$ and $\xi \in \mathfrak{C}_n$.

\begin{figure}[h]
\centering
\includegraphics[scale=0.4]{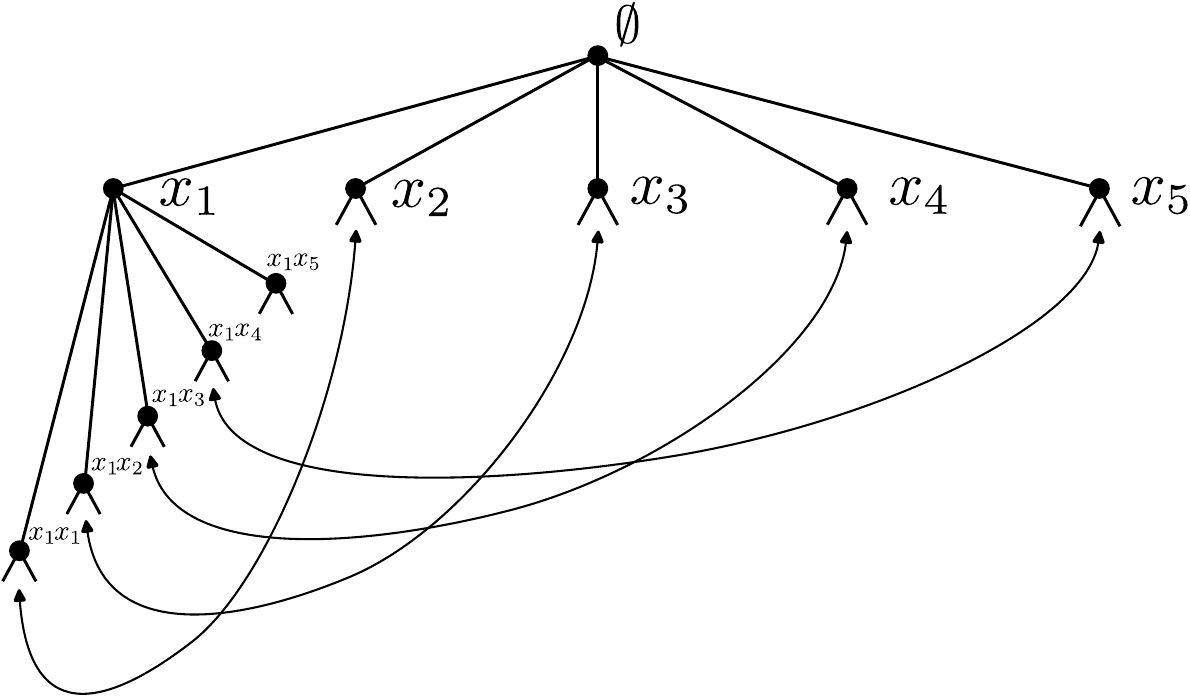}
\caption{The action of $\tau \in V_5$ on $\mathfrak{C}_5$.}
\end{figure}

\noindent The generating set $S_{\alpha}$ of $V_n$ that we want to construct consists of the elements $\dot{\sigma},\tau$ and one further involution $s_{\alpha}$, where $\alpha$ is a sequence in $V_n^{\ast} \defeq \coprod \limits_{\ell \in \N_0} V_n^{\ell}$.
The length of a sequence $\alpha = (\alpha_1,\ldots,\alpha_{\ell}) \in V_n^{\ast}$ will be denoted by $|\alpha| = \ell$.

\begin{definition}\label{def:s-alpha}
For each sequence $\alpha \in V_n^{\ast}$ of length $\ell$, we define $s_{\alpha} \in V_n$ as the element that is given by
\[
s_{\alpha}(x_{1}^{k} x_{i} \xi) = 
\begin{cases}
x_{1}^{k} x_{i} \xi, \text{ if $k=0$ and $i \neq 1$}\\
x_{1}^{k} x_{i} \alpha_{k}(\xi), \text{ if } 1 \leq k \leq \ell \text{ and } i = 2\\
x_{1}^{k} x_{i} \xi, \text{ if } 1 \leq k \leq \ell \text{ and } i \geq 3\\
x_{1}^{\ell+1} \dot{\sigma}(\xi), \text{ if } k = \ell \text{ and } i = 1
\end{cases}
\]
for $ \xi \in \mathfrak{C}_n$, and $x_i \in X_n$. 
\end{definition}

\begin{figure}[h]
\centering
\includegraphics[scale=0.55]{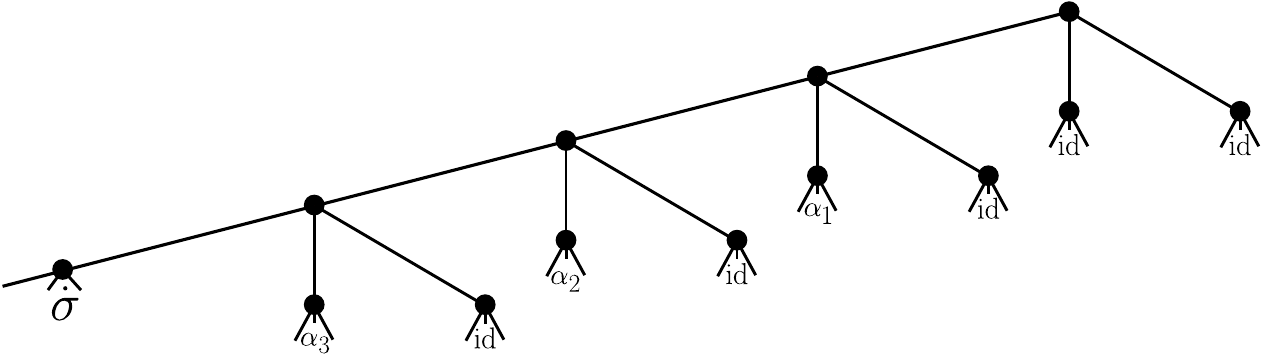}
\caption{The action of $s_{\alpha} \in V_3$ on $\mathfrak{C}_3$, where $\alpha = (\alpha_1,\alpha_2,\alpha_3)$.}
\end{figure}

\noindent Using Notation~\ref{not:gamma-vert-v}, we can write $s_{\alpha}$ as
\begin{equation}\label{eq:rewrite-s-alpha}
s_{\alpha} = x_1^{\ell+1} \dot{\sigma} \cdot \prod \limits_{k=1}^{\ell} x_1^k x_2 \alpha_{k}.
\end{equation}

\noindent Note that $s_{\alpha}$ is an involution if and only if $\alpha = (\alpha_i)_{i=1}^{\ell}$ consists of elements $\alpha_i$ of order at most $2$.
Our goal is to find conditions on $\alpha$ under which the group $H_{\alpha}$ generated by $S_{\alpha} \defeq \{\dot{\sigma},\tau,s_{\alpha}\}$ coincides with $V_n$.

\begin{remark}\label{rem:inspiration-s-alpha}
The definition of $s_{\alpha}$ is inspired by the class of tree automorphisms of `spinal type' that often appear in the theory of branch groups and automata groups, see e.g.~\cite[Definition 2.1.1]{BartholdiGrigorchukSunik03} for a definition and~\cite{KionkeSchesler23amenable,KionkeSchesler22} for recent constructions in which spinal tree automorphisms played a crucial role.
\end{remark}

\begin{notation}\label{not:commutator-and-conjugate}
Let $G$ be a group and let $g,h\in G$.
We write $[g,h] = ghg^{-1}h^{-1}$ and $g^h = h^{-1}gh$ to denote the commutator of $g$ and $h$, respectively the conjugate of $g$ by $h$.
\end{notation}

\begin{definition}\label{def:commutator-stable}
A generating set $S$ of a group $G$ is said to be \emph{commutator stable} if $G$ is also generated by $\Set{[s,t]}{s,t \in S}$.
\end{definition}

\begin{lemma}\label{lem:Vn-is-gen-by-invol}
The group $V_n'$ admits a finite commutator stable generating set $S$ that consists of involutions.
\end{lemma}
\begin{proof}
Let $\alpha \in V_n'$ be an involution, e.g.\ $\alpha = x_1 \dot{\sigma} \cdot x_2 \dot{\sigma}$, and let $\gamma \in V_n'$ be such that $s \defeq [\alpha,\alpha^{\gamma}] \neq \id$.
Since $V_n'$ is a simple group, see e.g.~\cite[Theorem 5.4]{Hig74}, it follows that $V_n'$ is generated by the conjugacy class $C$ of $s$.
Using that $V_n$, and hence $V_n'$, is finitely generated, we deduce that $V_n'$ admits a finite subset $A$ such that $\Set{s^g}{g \in A}$ generates $V_n'$.
Now the claim follows by setting $S = \Set{\alpha^g}{g \in A} \cup \Set{\alpha^{\gamma g}}{g \in A}$.
\end{proof}

\noindent Note that the element $t \defeq \dot{\sigma} \tau$ satisfies
\[
t(x_1x_1\xi) = \sigma (x_2\xi) = x_1\xi
\]
for every $\xi \in \mathfrak{C}_n$ and hence acts as a `translation' along the infinite word $x_1^{\infty} \in \mathfrak{C}_n$.
As a consequence we obtain
\begin{equation}\label{eq:conj-general-by-t}
(x_1^k\gamma)^t = x_1^{k+1}\gamma
\end{equation}
for every $k \in \N$ and every $\gamma \in V_n$.

\begin{definition}\label{def:udp}
Let $I$ be a subset of $\N$ and let $\mathcal{P}_2(I)$ denote the set of $2$-element subsets of $I$.
We say that $I$ has the \emph{unique difference property} if the map
\[
d \colon \mathcal{P}_2(I) \rightarrow \N,\ \{i,j\} \mapsto |i-j|
\]
is injective.
\end{definition}

\begin{remark}\label{rem:existence-udp}
Note that for every $N \in \N_0 \cup \{\infty\}$, there is a subset $I \subseteq \N$ of cardinality $N$ that satisfies the unique difference property.
For example every subset of $\Set{2^i}{i \in \N}$ has this property.
\end{remark}

\noindent Let $\alpha = (\alpha_1,\ldots,\alpha_{\ell}) \in V_n^{\ast}$.
By combining~\eqref{eq:rewrite-s-alpha} with~\eqref{eq:conj-general-by-t} we obtain
\begin{equation}\label{eq:s-alpha-to-t}
s_{\alpha}^{t^k}
= (x_1^{\ell+1} \dot{\sigma})^{t^k} \cdot (\prod \limits_{i=1}^{\ell} x_1^i x_2 \alpha_{i})^{t^k}
= x_1^{\ell+k+1} \dot{\sigma} \cdot \prod \limits_{i=1}^{\ell} x_1^{i+k} x_2 \alpha_{i}
\end{equation}
for every $k \in \N_0$.
In view of~\eqref{eq:s-alpha-to-t}, we can easily deduce the following.

\begin{lemma}\label{lem:commutator-trick}
Let $\alpha = (\alpha_1,\ldots,\alpha_{\ell}) \in V_n^{\ast}$.
Suppose that there is a number $k \in \N_0$ such that $\alpha_{i} = \id$ for every $\ell-k+1 \leq i \leq \ell$.
Then
\[
[s_{\alpha},s_{\alpha}^{t^k}]
= x_1^{\ell+1} [\dot{\sigma},x_1^k \dot{\sigma}]
\cdot \prod \limits_{i=k+1}^{\ell} x_1^i x_2[\alpha_i,\alpha_{i-k}].
\]
\end{lemma}
\begin{proof}
Since $\alpha_{i} = 1$ for every $\ell-k+1 \leq i \leq \ell$, it follows that $x_1^{i+k} x_2 \alpha_{i} = \id$ if $i+k \geq \ell+1$.
Thus $x_1^{\ell+1} \dot{\sigma}$ commutes with $x_1^{i+k} x_2 \alpha_{i}$ for every $i$ and we can apply~\eqref{eq:s-alpha-to-t} to deduce that
\begin{align*}
[s_{\alpha},s_{\alpha}^{t^k}]
&= [x_1^{\ell+1} \dot{\sigma} \cdot \prod \limits_{i=1}^{\ell} x_1^i x_2 \alpha_{i},
x_1^{\ell+k+1} \dot{\sigma} \cdot \prod \limits_{i=1}^{\ell} x_1^{i+k} x_2 \alpha_{i}]\\
&= [x_1^{\ell+1} \dot{\sigma},x_1^{\ell+k+1} \dot{\sigma}] \cdot
[\prod \limits_{i=1}^{\ell} x_1^i x_2 \alpha_{i}, \prod \limits_{i=1}^{\ell} x_1^{i+k} x_2 \alpha_{i}]\\
&= x_1^{\ell+1} [\dot{\sigma},x_1^k \dot{\sigma}]
\cdot \prod \limits_{i=k+1}^{\ell} x_1^i x_2[\alpha_i,\alpha_{i-k}].
\end{align*}
\end{proof}

\subsection*{Proving the main theorem}
For the rest of the section we fix a sequence $\alpha = (\alpha_1,\ldots,\alpha_{\ell}) \in V_n^{\ast}$ such that
\begin{enumerate}
\item $\{\alpha_1,\ldots,\alpha_{\ell}\}$ is a commutator stable generating set of $V_n'$,
\item each $\alpha_i$ is of order at most $2$,
\item $I \defeq \Set{1 \leq i \leq \ell}{\alpha_i \neq \id}$ satisfies the unique difference property,
\item $\alpha_i = \id$ for every $i \leq N \defeq \max \Set{\abs{i-j}}{\{i,j\} \in \mathcal{P}_2(I)}$,
\item $\alpha_{\ell-i} = \id$ for every $0 \leq i \leq N$.
\end{enumerate}

\noindent The existence of such a sequence $\alpha$ is guaranteed by Lemma~\ref{lem:Vn-is-gen-by-invol} and Remark~\ref{rem:existence-udp}.
We will show that the group $H_{\alpha}$ that is generated by $S_{\alpha} = \{\dot{\sigma},\tau,s_{\alpha}\}$ coincides with $V_n$.
Let us start with some auxiliary observations.
Recall that we write $E_n$ to denote the subgroup of volume-preserving homeomorphisms of $V_n$.

\begin{lemma}\label{lem:isolating-j-i-gamma}
Let $i,j \in I$ be numbers with $i < j$.
There is
an element $\gamma_{i,j} \in E_n$ such that
\[
[s_{\alpha},s_{\alpha}^{t^{j-i}}]
= x_1^{\ell+1} \gamma_{i,j}
\cdot x_1^j x_2[\alpha_j,\alpha_{i}].
\]
\end{lemma}
\begin{proof}
Let $k \defeq j-i$.
From our assumption on $\alpha$ we know that $N \geq k$ and that $\alpha_{i} = \id$ for every $\ell-N \leq i \leq \ell$.
In particular we have $\alpha_{i} = \id$ for every $\ell-k+1 \leq i \leq \ell$ so that we can apply Lemma~\ref{lem:commutator-trick} to deduce that
\begin{align*}
[s_{\alpha},s_{\alpha}^{t^k}]
&= x_1^{\ell+1} [\dot{\sigma},x_1^k \dot{\sigma}]
\cdot \prod \limits_{m=k+1}^{\ell} x_1^m x_2[\alpha_m,\alpha_{m-k}]\\
&= x_1^{\ell+1} [\dot{\sigma},x_1^{j-i} \dot{\sigma}]
\cdot x_1^{j} x_2[\alpha_j,\alpha_{i}] \cdot \prod_{\substack{m=j-i+1\\ m \neq j}}^{\ell} x_1^m x_2[\alpha_m,\alpha_{m-(j-i)}]\\
&= x_1^{\ell+1} \gamma_{i,j} \cdot x_1^j x_2[\alpha_j,\alpha_{i}],
\end{align*}
where the latter equality follows by setting $\gamma_{i,j} = [\dot{\sigma},x_1^{j-i} \dot{\sigma}]$ and by applying our assumption that $I$ satisfies the unique difference property.
\end{proof}

\begin{lemma}\label{lem:simple-times-finite}
Let $Q$ be an infinite simple group, let $F$ be a finite group and let $G$ be a subgroup of $Q \times F$.
If $G$ is subdirect in $Q \times F$, then $G$ coincides with $Q \times F$.
\end{lemma}
\begin{proof}
Let $G$ be subdirect in $Q \times F$.
By definition this means that the canonical projections $\pi_F \colon G \rightarrow F$ and $\pi_Q \colon G \rightarrow Q$ are surjective.
Let $K$ denote the kernel of $\pi_F$.
Then $K$ is a normal finite index subgroup of the form $K_0 \times \id$ in $G$.
In view of the surjectivity of $\pi_Q$ is follows that $G$ is infinite and that $K \cong K_0$ is a non-trivial normal subgroup of $Q$.
Since $Q$ is simple it follows that $K_0$ coincides with $Q$.
Since $G$ is subdirect this implies that $G = Q \times F$.
\end{proof}

\begin{lemma}\label{lem:containing-x_1x_2V_n}
The group $H_{\alpha}$ contains $x_1^{i_0}x_2 V_n'$, where $i_0 \defeq \max (I)$.
\end{lemma}
\begin{proof}
For any two pairwise distinct elements $i,j \in I$, let $s_{i,j} \defeq [\alpha_i,\alpha_j]$.
In the case where $i < j$ we know from Lemma~\ref{lem:isolating-j-i-gamma} that $x_1^{\ell+1} \gamma_{i,j}
\cdot x_1^j x_2 s_{j,i} \in H_{\alpha}$ for some element $\gamma_{i,j} \in E_n$.
Using the formula $[g,h]^{-1} = [h,g]$, we may therefore deduce that $x_1^{\ell+1} \gamma_{i,j}^{-1}
\cdot x_1^j x_2 s_{i,j}$ lies in $H_{\alpha}$.
In view of~\eqref{eq:conj-general-by-t} this shows that $x_1^{\ell+1+(i_0-j)} \gamma_{i,j}^{-1}
\cdot x_1^{i_0} x_2 s_{i,j}$ lies in $H_{\alpha}$.
In particular it follows that for every distinct elements $i,j \in I$ there is an element $\gamma_{i,j}' \in E_n$ with
\[
s_{i,j}' \defeq x_1^{\ell+1} \gamma_{i,j}'
\cdot x_1^{i_0} x_2 s_{i,j} \in H_{\alpha}.
\]
Note that the element $x_1^{\ell+1} \gamma_{i,j}'$ and $x_1^{i_0} x_2 s_{i,j}$ commute since $\ell+1 > i_0$.
Since $\{\alpha_1,\ldots,\alpha_{\ell}\}$, and hence $\Set{\alpha_i}{i \in I}$, is a commutator stable generating set of $V_n'$ it therefore follows that for each $\alpha \in V_n'$ there is an element $\gamma_{\alpha} \in E_n$ with $x_1^{\ell+1} \gamma_{\alpha}
\cdot x_1^{i_0} x_2 \alpha \in H_{\alpha}$.
Thus $H_{\alpha}$ contains a subdirect product of $x_1^{i_0}x_2 V_n'$ and $x_1^{\ell+1} F$, where $F$ is a finitely generated subgroup of $E_n$.
Since $F$ is finite by Lemma~\ref{lem:level-pres}, the claim now follows from Lemma~\ref{lem:simple-times-finite}.
\end{proof}

\begin{lemma}\label{lem:sum-in-V_n}
The group $H_{\alpha}$ contains $x_i V_n$ for every $x_i \in X_n$.
\end{lemma}
\begin{proof}
From Lemma~\ref{lem:containing-x_1x_2V_n} we know that $H_{\alpha}$ contains $x_1^{i_0}x_2 V_n'$, where $i_0 \defeq \max(I)$.
In view of~\eqref{eq:conj-general-by-t} we see that we can conjugate $x_1^{i_0}x_2 V_n'$ with an appropriate power of $t$ to obtain $x_1^{k}x_2 V_n' \leq H_{\alpha}$ for every $k \in \N$.
In particular, it follows that $\prod \limits_{k=1}^{\ell} x_1^k x_2 \alpha_{k}$ lies in $H_{\alpha}$.
On the other hand we have $s_{\alpha} = x_1^{\ell+1} \dot{\sigma} \cdot \prod \limits_{k=1}^{\ell} x_1^k x_2 \alpha_{k} \in H_{\alpha}$ by the definition of $s_{\alpha}$ so that we can deduce that $x_1^{\ell+1} \dot{\sigma}$ lies in $H_{\alpha}$.
We can therefore conjugate $x_1^{\ell+1}x_2 V_n'$ with $x_1^{\ell+1} \dot{\sigma}$ to obtain $x_1^{\ell+2} V_n' \leq H_{\alpha}$.
Since by Theorem~\ref{thm:simple-commutator} the abelianization $V_n^{\ab}$ of $V_n$ is generated by the image of $\dot{\sigma}$, it follows that
\[
\langle x_1^{\ell+2} V_n', (x_1^{\ell+1} \dot{\sigma})^t \rangle
= \langle x_1^{\ell+2} V_n', x_1^{\ell+2} \dot{\sigma} \rangle
= x_1^{\ell+2} \langle V_n', \dot{\sigma} \rangle
= x_1^{\ell+2} V_n
\]
is a subgroup of $H_{\alpha}$.
By conjugating this subgroup with $t^{-\ell-1}$ we obtain $x_1 V_n \leq H_{\alpha}$ and hence $x_1x_i V_n \leq H_{\alpha}$ for every $x_i \in X_{n}$.
Finally, we can conjugate the groups $x_1x_i V_n$, where $i < n$, with $\tau$ to conclude that $x_iV_n \leq H_{\alpha}$ for every $x_i \in X_n$, which completes the proof.
\end{proof}

\begin{lemma}\label{lem:sym-in-V_n}
The group $H_{\alpha}$ contains $\dot{\Sym}(X_n)$.
\end{lemma}
\begin{proof}
From Lemma~\ref{lem:sum-in-V_n} it follows that $x_1 \dot{\Sym}(X_n)$ is a subgroup of $H_{\alpha}$.
In particular it follows that the group $x_1 \dot{\Sym}(X_n \setminus \{x_n\})$ and its conjugate
\[
(x_1 \dot{\Sym}(X_n \setminus \{x_n\}))^{\tau} = \dot{\Sym}(X_n \setminus \{x_1\})
\]
lie in $H_{\alpha}$.
Since $\langle \sigma, \Sym(X_n \setminus \{x_1\}) \rangle = \Sym(X_n)$, we deduce that $\dot{\Sym}(X_n) \leq H_{\alpha}$.
\end{proof}

\begin{theorem}\label{thm:main}
The group $H_{\alpha}$ coincides with $V_n$.
\end{theorem}
\begin{proof}
From Lemma~\ref{lem:sum-in-V_n} and Lemma~\ref{lem:sym-in-V_n} we know that $H_{\alpha}$ contains $\oplus_{i=1}^n x_i V_n \rtimes \dot{\Sym(X_n)}$, which is a maximal subgroup of $V_n$ by Proposition~\ref{prop:maximal}.
Since $\tau \in H_{\alpha}$ is not contained in $\oplus_{i=1}^n x_i V_n \rtimes \dot{\Sym(X_n)}$, it therefore follows that $H_{\alpha}$ coincides with $V_n$.
\end{proof}

\bibliographystyle{alpha}
\bibliography{literature.bib}

\end{document}